\documentclass[12pt,a4paper]{article}
\usepackage{amsmath,enumerate,bm,color,geometry}
\usepackage{amsthm}
\usepackage{amssymb}
\usepackage{amsbsy}
\usepackage{amsfonts}
\usepackage{amstext}
\usepackage{amscd}
\usepackage{graphicx}
\numberwithin{equation}{section}
\theoremstyle{plain}
\newtheorem{thm}{Theorem}[section]
\newtheorem{prop}[thm]{Proposition}
\newtheorem{cor}[thm]{Corollary}
\newtheorem{lem}[thm]{Lemma}
\theoremstyle{definition}

\newtheorem{rem}[thm]{Remark}
\newtheorem{defi}[thm]{Definition}

\newcommand{\R}{\mathbb{R}}
\newcommand{\C}{\mathbb{C}}

\newcommand{\im}{\text{\normalfont Im}}

\newcommand\supp{{\rm supp}}
\newcommand{\FC}{\mathcal{C}^\boxplus}
\newcommand{\dd}{{\rm d}} % for integral
\newcommand{\ii}{{\rm i}}  % A solution of z^2=-1 
\newcommand\ep{\varepsilon} 

\begin{document}
\title{On free generalized inverse gaussian distributions}

\author{Takahiro Hasebe\\Department of Mathematics,\\ Hokkaido University \\thasebe@math.sci.hokudai.ac.jp \and
	Kamil  Szpojankowski \\ Faculty of Mathematics and Information Science\\ Warsaw University of Technology\\
	k.szpojankowski@mini.pw.edu.pl}
%\email{}

\date{\today}

\maketitle

\begin{abstract}
	We study here properties of {\it free Generalized Inverse Gaussian distributions} (fGIG) in free probability. We show that in many cases the fGIG shares similar properties with the classical GIG distribution. In particular we prove that fGIG is freely infinitely divisible, free regular and unimodal, and moreover we determine which distributions in this class are freely selfdecomposable. In the second part of the paper we prove that for free random variables $X,Y$ where $Y$ has a free Poisson distribution one has $X\stackrel{d}{=}\frac{1}{X+Y}$ if and only if $X$ has fGIG distribution for special choice of parameters. We also point out that the free GIG distribution maximizes the same free entropy functional as the classical GIG does for the classical entropy.
\end{abstract}

%%%%%%%%%%%%%%%%%%%
%%%%%%%%%%%%%%%%%%%
\section{Introduction}

%What is free Probability.

Free probability was introduced by Voiculescu in \cite{Voi85} as a non-commutative probability theory where one defines a new notion of independence, so called freeness or free independence. Non-commutative probability is a counterpart of the classical probability theory where one allows random variables to be non-commutative objects. Instead of defining a probability space as a triplet $(\Omega,\mathcal{F},\mathbb{P})$ we switch to a pair $(\mathcal{A},\varphi)$ where $\mathcal{A}$ is an algebra of random variables and $\varphi\colon\mathcal{A}\to\mathbb{C}$ is a linear functional, in classical situation $\varphi=\mathbb{E}$. It is natural then to consider algebras $\mathcal{A}$ where random variables do not commute (for example $C^*$ or $W^*$--algebras). For bounded random variables independence can be equivalently understood as a rule of calculating mixed moments. It turns out that while for commuting random variables only one such rule leads to a meaningful notion of independence, the non-commutative setting is richer and one can consider several notions of independence. Free independence seems to be the one which is the most important. The precise definition of freeness is stated in Section 2 below.

%What are analogies between free and classical probability. (BP bijections, characterizations)

Free probability emerged from questions related to operator algebras however the development of this theory showed that it is surprisingly closely related with the classical probability theory. First evidence of such relations appeared with Voiculescu's results about asymptotic freeness of random matrices. Asymptotic freeness roughly speaking states that (classically) independent, unitarily invariant random matrices, when size goes to infinity, become free. \\
Another link between free and classical probability goes via infinite divisibility. With a notion of independence in hand one can consider a convolution of probability measures related to this notion. For free independence such operation is called free convolution  and it is denoted by $\boxplus$. More precisely for free random variables $X,Y$ with respective distributions $\mu,\nu$ the distribution of the sum $X+Y$ is called the free convolution of $\mu$ and $\nu$ and is denoted by $\mu\boxplus\nu$. The next natural step is to ask which probability measures are infinitely divisible with respect to this convolution. We say that $\mu$ is freely infinitely divisible if for any $n \geq 1$ there exists a probability measure $\mu_n$ such that
\[
\mu=\underbrace{\mu_n\boxplus\ldots\boxplus\mu_n}_{\text{$n$ times}}.
\]
Here we come across another striking relation between free and classical probability: there exists a bijection between classically and freely infinitely divisible probability measures, this bijection was found in \cite{BP99} and it is called Bercovici-Pata (BP) bijection. This bijection has number of interesting properties, for example measures in bijection have the same domains of attraction. In free probability literature it is standard approach to look for the free counterpart of a classical distribution via BP bijection. For example Wigner's semicircle law plays the role of the Gaussian law in free analogue of Central Limit Theorem, Marchenko-Pastur distribution appears in the limit of free version of Poisson limit theorem and is often called free Poisson distribution.

While BP bijection proved to be a powerful tool, it does not preserve all good properties of distributions. Consider for example Lukacs theorem which says that for classically independent random variables $X,Y$ random variables $X+Y$ and $X/(X+Y)$ are independent if and only if $X,Y$ have gamma distribution with the same scale parameter \cite{Luk55}. One can consider similar problem in free probability and gets the following result (see \cite{Szp15,Szp16}) for free random variables $X,Y$ random variables $X+Y$ and $(X+Y)^{-1/2}X(X+Y)^{-1/2}$ are free if and only if $X,Y$ have Marchenko-Pastur (free Poisson) distribution with the same rate. From this example one can see our point - it is not the image under BP bijection of the Gamma distribution (studied in \cite{PAS08,HT14}), which has the Lukacs independence property in free probability, but in this context free Poisson distribution plays the role of the classical Gamma distribution.

%Where fGIG appeared for the first time. It is natural to study further properties of this distribution.

In \cite{Szp17} another free independence property was studied -- a free version of so called Matsumoto-Yor property (see \cite{MY01,LW00}). In classical probability this property says that for independent $X,Y$ random variables $1/(X+Y)$ and $1/X-1/(X+Y)$ are independent if and only if $X$ has a Generalized Inverse Gaussian (GIG) distribution and $Y$ has a Gamma distribution. In the free version of this theorem (i.e. the theorem where one replaces classical independence assumptions by free independence) it turns out that the role of the Gamma distribution is taken again by the free Poisson distribution and the role of the GIG distribution plays a probability measure which appeared for the first time in \cite{Fer06}. We will refer to this measure as the free Generalized Inverse Gaussian distribution or fGIG for short. We give the definition of this distribution in Section 2.

The main motivation of this paper is to study further properties of fGIG distribution. The results from \cite{Szp17} suggest that in some sense (but not by means of the BP bijection) this distribution is the free probability analogue of the classical GIG distribution. It is natural then to ask if fGIG distribution shares more properties with its classical counterpart. It is known that the classical GIG distribution is infinitely divisible (see \cite{BNH77}) and selfdecomposable (see \cite{Hal79,SS79}). In \cite{LS83} the GIG distribution was characterized in terms of an equality in distribution, namely if we take $X,Y_1,Y_2$ independent and such that $Y_1$ and $Y_2$ have Gamma distributions with suitable parameters and we assume that 
\begin{align} 
X\stackrel{d}{=}\frac{1}{Y_2+\frac{1}{Y_1+X}}
\end{align}
then $X$ necessarily has a GIG distribution. A simpler version of this theorem characterizes smaller class of fGIG distributions by equality
\begin{align} \label{eq:char}
X\stackrel{d}{=}\frac{1}{Y_1+X}
\end{align}
for $X$ and $Y_1$ as described above.

The overall result of this paper is that the two distributions GIG and fGIG indeed have many similarities. We show that fGIG distribution is freely infinitely divisible and even more that it is free regular. Moreover fGIG distribution can be characterized by the equality in distribution \eqref{eq:char}, where one has to replace the independence assumption by freeness and assume that $Y_1$ has free Poisson distributions. While there are only several examples of freely selfdecomposable distributions it is interesting to ask whether fGIG has this property. It turns out that selfdecomposability is the point where the symmetry between GIG and fGIG partially breaks down: not all fGIG distributions are freely selfdecomposable. We find conditions on the parameters of fGIG family for which this distributions are freely selfdecomposable. 
Except from the results mentioned above we prove that fGIG distribution is unimodal. We also point out that in \cite{Fer06} it was proved that fGIG maximizes a certain free entropy functional. An easy application of Gibbs' inequality shows that the classical GIG maximizes the same functional of classical entropy.

The paper is organized as follows: In Section 2 we shortly recall basics of free probability and next we study some properties of fGIG distributions. Section 3 is devoted to the study of free infinite divisibility, free regularity, free selfdecomposability and unimodality of the fGIG distribution. In Section 4 we show that the free counterpart of the characterization of GIG distribution by \eqref{eq:char} holds true, and we discuss entropy analogies between GIG and fGIG.

%%%%%%%%%%%%%%%%%%%%%%%%%%
%%%%%%%%%%%%%%%%%%%%%%%%%%%
\section{Free GIG distributions}
In this section we recall the definition of free GIG distribution and study basic properties of this distribution. In particular we study in detail the $R$-transform of fGIG distribution. Some of the properties established in this section will be crucial in the subsequent sections where we study free infinite divisibility of the free GIG distribution and characterization of the free GIG distribution.
The free GIG distribution appeared for the first time (not under the name free GIG) as the almost sure weak limit of empirical spectral distribution of GIG matrices (see \cite{Fer06}).

%%%%%%%%%%%%%%%%%%%%%%%
\subsection{Basics of free probability}
This paper deals mainly with properties of free GIG distribution related to free probability and in particular to free convolution. Therefore in this section we introduce notions and tools that we need in this paper. The introduction is far from being detailed, reader not familiar with free probability may find a very good introduction to the theory in \cite{VDN92,NS06,MS}.

\begin{enumerate}[$1^o$]
	\item A $C^*$--probability space is a pair $(\mathcal{A},\varphi)$, where $\mathcal{A}$ is a unital $C^*$-algebra and $\varphi$ is a linear functional $\varphi\colon\mathcal{A}\to\mathbb{C}$, such that $\varphi(\mathit{1}_\mathcal{A})=1$ and $\varphi(aa^*)\geq 0$. Here by $\mathit{1}_\mathcal{A}$ we understand the unit of $\mathcal{A}$.
	
	\item Let $I$ be an index set. A family of subalgebras $\left(\mathcal{A}_i\right)_{i\in I}$ are called free if $\varphi(X_1\cdots X_n)=0$ whenever $a_i\in \mathcal{A}_{j_i}$, $j_1\neq j_2\neq \ldots \neq j_n$ and $\varphi(X_i)=0$ for all $i=1,\ldots,n$ and $n=1,2,\ldots$. Similarly, self-adjoint random variables $X,\,Y\in\mathcal{A}$ are free (freely independent) when subalgebras generated by $(X,\,\mathit{1}_\mathcal{A})$ and $(Y,\,\mathit{1}_\mathcal{A})$ are freely independent.
	
	\item The distribution of a self-adjoint random variable is identified via moments, that is for a random variable $X$ we say that a probability measure $\mu$ is the distribution of $X$ if \[\varphi(X^n)=\int t^n\,\dd\mu(t),\,\mbox{for all } n=1,2,\ldots\]
	Note that since we assume that our algebra $\mathcal{A}$ is a $C^*$--algebra, all random variables are bounded, thus the sequence of moments indeed determines a unique probability measure.
	
	\item The distribution of the sum $X+Y$ for free random variables $X,Y$ with respective distributions $\mu$ and $\nu$ is called the free convolution of $\mu$ and $\nu$, and is denoted by $\mu\boxplus\nu$.
\end{enumerate}

%%%%%%%%%%%%%%%%%%%%%
\subsection{Free GIG distribution}
In this paper we are concerned with a specific family of probability measures which we will refer to as free GIG (fGIG) distributions.
\begin{defi}
	The free Generalized Inverse Gaussian (fGIG) distribution is a measure $\mu=\mu(\alpha,\beta,\lambda)$, where $\lambda\in\mathbb{R}$ and $\alpha,\beta>0$ which is compactly supported on the interval $[a,b]$ with the density
	\begin{align*}
	\mu(\dd x)=\frac{1}{2\pi}\sqrt{(x-a)(b-x)} \left(\frac{\alpha}{x}+\frac{\beta}{\sqrt{ab}x^2}\right)\dd x,
	\end{align*}
	where $0<a<b$ are the solution of
	\begin{align}\label{eq1}
	1-\lambda+\alpha\sqrt{ab}-\beta\frac{a+b}{2ab}=&0\\
	\label{eq2}
	1+\lambda+\frac{\beta}{\sqrt{ab}}-\alpha\frac{a+b}{2}=&0.
	\end{align}
\end{defi}

Observe that the system of equations for coefficients for fixed $\lambda\in \mathbb{R}$ and $\alpha,\beta>0$ has a unique solution $0<a<b$. We can easily get the following
\begin{rem}\label{prop:ab}
Let $\lambda\in\R$. Given $\alpha,\beta>0$, the system of equations \eqref{eq1}, \eqref{eq2}
has a unique solution $(a,b)$ such that 
\begin{equation}\label{unique}
0<a<b, \qquad |\lambda| \left(\frac{\sqrt{a}-\sqrt{b}}{\sqrt{a}+\sqrt{b}}\right)^2<1.
\end{equation} 
Conversely, given $(a,b)$ satisfying \eqref{unique}, the set of equations \eqref{eq1}--\eqref{eq2} has a unique solution $(\alpha,\beta)$, which is given by   
\begin{align}
&\alpha = \frac{2}{(\sqrt{a}-\sqrt{b})^2}\left( 1 + \lambda \left(\frac{\sqrt{a}-\sqrt{b}}{\sqrt{a}+\sqrt{b}}\right)^2 \right) >0, \label{eq3}\\
&\beta = \frac{2 a b}{ (\sqrt{a} - \sqrt{b})^2}\left( 1 - \lambda \left(\frac{\sqrt{a}-\sqrt{b}}{\sqrt{a}+\sqrt{b}}\right)^2 \right)>0. \label{eq4}
\end{align}
\end{rem}
Thus we may parametrize fGIG distribution using parameters $(a,b,\lambda)$ satisfying \eqref{unique} instead of $(\alpha,\beta,\lambda)$. We will make it clear whenever we will use a parametrization different than $(\alpha,\beta,\lambda)$. 
\begin{rem}
	It is useful to introduce another parameterization to describe the distribution $\mu(\alpha,\beta,\lambda)$. Define 
	\begin{equation}\label{eq:AB}
	A=(\sqrt{b}-\sqrt{a})^2, \qquad B= (\sqrt{a}+\sqrt{b})^2, 
	\end{equation}
	observe that we have then 
	\begin{align*}
	\alpha =& \frac{2}{A}\left( 1 + \lambda \frac{A}{B} \right) >0, \qquad
	\beta = \frac{(B-A)^2}{8 A}\left( 1 - \lambda \frac{A}{B} \right)>0,\\
	a =& \left(\frac{\sqrt{B}-\sqrt{A}}{2}\right)^2,\qquad
	b = \left(\frac{\sqrt{A}+\sqrt{B}}{2}\right)^2.
	\end{align*}
	The condition \eqref{unique} is equivalent to 
	\begin{equation}\label{eq:ABineq}
	0<\max\{1,|\lambda|\}A<B.
	\end{equation}
	Thus one can describe any measure $\mu(\alpha,\beta,\lambda)$ in terms of $\lambda,A,B$.
\end{rem}

%%%%%%%%%%%%%%%%%%%%%%%%%%%%%%%%
\subsection{$R$-transform of fGIG distribution}\label{sec:form}

The $R$-transform of the measure $\mu(\alpha,\beta,\lambda)$ was calculated in \cite{Szp17}. Since the $R$-transform will play a crucial role in the paper we devote this section for a detailed study of its properties. We also point out some properties of fGIG distribution which are derived from properties of the $R$-transform.

Before we present the $R$-transform of fGIG distribution let us briefly recall how the $R$-transform is defined and stress its importance for free probability.
\begin{rem}\label{rem:Cauchy}
\begin{enumerate}[$1^o$]
	\item For a probability measure $\mu$ one defines its Cauchy transform via 
	\[G_\mu(z)=\int \frac{1}{z-x}\dd\mu(x).\]
	It is an analytic function on the upper-half plane with values in the lower half-plane. Cauchy transform determines uniquely the measure and there is an inversion formula called Stieltjes inversion formula, namely for $h_\ep(t)=-\tfrac{1}{\pi}\im\, G_\mu(t+i\ep)$ one has
	\[
	\dd\mu(t)=\lim_{\ep\to 0^+} h_\ep(t)\,\dd t, 
	\]
	where the limit is taken in the weak topology.
	
	\item For a compactly supported measure $\mu$ one can define in a neighbourhood of the origin so called $R$-transform  by
	\[R_\mu(z)=G_\mu^{\langle -1 \rangle}(z)-\frac{1}{z},\]
	where by $G_\mu^{\langle -1 \rangle}$ we denote the inverse under composition of the Cauchy transform of $\mu$.\\
	The relevance of the $R$-transform for free probability comes form the fact that it linearizes free convolution, that is $R_{\mu\boxplus\nu}=R_\mu+R_\nu$ in a neighbourhood of zero. 
	\end{enumerate}
\end{rem}
The $R$-transform of fGIG distribution is given by
\begin{equation}\label{eq:R_F}
\begin{split}
r_{\alpha,\beta,\lambda}(z) 
&= \frac{-\alpha + (\lambda+1)z + \sqrt{f_{\alpha,\beta,\lambda}(z)}}{2z(\alpha-z)} 
\end{split}
\end{equation}
in a neighbourhood of $0$, where the square root is the principal value, 
\begin{equation}\label{eq:pol_fpar} 
f_{\alpha,\beta,\lambda}(z)=(\alpha+(\lambda-1)z)^2-4\beta z (z-\alpha)(z-\gamma),
\end{equation}
and 
\begin{align*}
\gamma=\frac{\alpha^2 a b+\frac{\beta^2}{ab}-2\alpha\beta\left(\frac{a+b}{\sqrt{ab}}-1\right)-(\lambda-1)^2}{4\beta}.
\end{align*}
Note that $z=0$ is a removable singular point of $r_{\alpha,\beta,\lambda}$. Observe that in terms of $A,B$ defined by \eqref{eq:AB} we have
\begin{align*}
\gamma &= 2\frac{\lambda A^2 + A B - 2B^2}{B(B-A)^2}. 
\end{align*}
It is straightforward to observe that \eqref{eq:ABineq} implies $A (\lambda A+B)<2 A B<2B^2$, thus we have $\gamma<0$. 

The following remark was used in \cite[Remark 2.1]{Szp17} without a proof.  We give a proof here. 
\begin{rem}\label{rem:F_lambda_sym}
	We have $f_{\alpha,\beta,\lambda}(z)=f_{\alpha,\beta,-\lambda}(z)$, where $\alpha,\beta>0,\lambda\in\mathbb{R}$. 
	\begin{proof}	
	To see this one has to insert the definition of $\gamma$ into \eqref{eq:pol_fpar} to obtain 	
	\[f_{\alpha,\beta,\lambda}(z)=\alpha z \lambda^2+\left(\left(ab\alpha^2-2\alpha\beta\frac{a+b}{\sqrt{ab}}+\frac{\beta^2}{ab}+2\alpha\beta\right)z-4 \beta z^2-\alpha\right) (z-\alpha), 
	\]
where $a=a(\alpha,\beta,\lambda)$ and $b=b(\alpha,\beta,\lambda)$.  Thus it suffices to show that the quantity $g(\alpha,\beta,\lambda):=ab\alpha^2-2\alpha\beta\frac{a+b}{\sqrt{ab}}+\frac{\beta^2}{ab}$ does not depend on the sign of $\lambda$.  To see this, observe from the system of equations \eqref{eq1} and \eqref{eq2} that $a(\alpha,\beta,-\lambda)=\frac{\beta}{\alpha b(\alpha,\beta,\lambda)}$ and  $b(\alpha,\beta,-\lambda)=\frac{\beta}{\alpha a(\alpha,\beta,\lambda)}$. It is then straightforward to check that 
$
g(\alpha,\beta,-\lambda) = g(\alpha,\beta,\lambda). 
$
	\end{proof}
\end{rem}

\begin{prop}
	The $R$-transform of the measure $\mu(\alpha,\beta,\lambda)$ can be extended to a function (still denoted by $r_{\alpha,\beta,\lambda}$) which is analytic on $\mathbb{C}^{-}$ and continuous on $(\C^- \cup\R)\setminus\{\alpha\}$.
\end{prop}
\begin{proof}	
A direct calculation shows that using parameters $A,B$ defined by \eqref{eq:AB} the polynomial $f_{\alpha,\beta,\lambda}$ under the square root factors as
\begin{equation*}
f_{\alpha,\beta,\lambda}(z) =  \frac{(B-A)^2(B-\lambda A)}{2 A B}\left[z +\frac{2(B+\lambda A)}{B(B-A)}\right]^2  \left[  \frac{2B}{A(B-\lambda A)} - z   \right]. 
\end{equation*}
Thus we can write 
\begin{equation}\label{eq:pol_par}
f_{\alpha,\beta,\lambda}(z) = 4\beta (z-\delta)^2(\eta-z), 
\end{equation}
where 
\begin{align}
&\delta = - \frac{2(B+\lambda A)}{B(B-A)}<0, \label{eq5}\\
&\eta = \frac{2B}{A(B-\lambda A)} >0. \label{eq6}
\end{align}
It is straightforward to verify that \eqref{eq:ABineq} implies $\eta \geq \alpha$ with equality valid only when $\lambda=0$.\\
Calculating $f_{\alpha,\beta,\lambda}(0)$ using first \eqref{eq:pol_fpar} and then \eqref{eq:pol_par} we get $4 \beta \eta \delta^2 = \alpha^2$, since $\eta \geq \alpha$ we see that $\delta \geq -\sqrt{\alpha/(4\beta)}$ with equality only when $\lambda=0$. 

Since all roots of $f_{\alpha,\beta,\lambda}$ are real, the square root $\sqrt{f_{\alpha,\beta,\lambda}(z)}$ may be defined continuously on $\C^-\cup\R$ so that $\sqrt{f_{\alpha,\beta,\lambda}(0)}=\alpha$. As noted above $\delta<0$, and continuity of $f_{\alpha,\beta,\lambda}$ implies that we have
 \begin{equation}\label{RRR}
 \sqrt{f_{\alpha,\beta,\lambda}(z)} = 2(z-\delta)\sqrt{\beta(\eta-z)},  
\end{equation}
where we take the principal value of the square root in the expression $\sqrt{4\beta(\eta-z)}$. Thus finally we arrive at the following form of the $R$-transform
\begin{equation}\label{R}
\begin{split}
r_{\alpha,\beta,\lambda}(z) 
&= \frac{-\alpha + (\lambda+1)z + 2(z-\delta)\sqrt{\beta(\eta-z)}}{2z(\alpha-z)} 
\end{split}
\end{equation}
which is analytic in $\C^-$ and continuous in $(\C^- \cup\R)\setminus\{\alpha\}$ as required. 
\end{proof}
Next we describe the behaviour of the $R$-transform around the singular point $z=\alpha$.

\begin{prop}
	If $\lambda>0$ then 
	\begin{equation}\label{alpha}
	\begin{split}
	r_{\alpha,\beta,\lambda}(z) 
	= \frac{\lambda}{\alpha-z} -\frac{1}{2\alpha} \left(1+\lambda+\frac{\sqrt{\beta}(2\eta-3\alpha+\delta)}{\sqrt{\eta-\alpha}}\right) + o(1),\qquad\mbox{as } z\to\alpha. 
	\end{split}
	\end{equation}
	If $\lambda<0$ then
	\begin{equation}\label{alpha2}
	r_{\alpha,\beta,\lambda}(z) = - \frac{1}{2\alpha}\left(1+\lambda+\frac{\sqrt{\beta}(2\eta-3\alpha+\delta)}{\sqrt{\eta-\alpha}}\right)+ o(1), \qquad \mbox{as } z\to\alpha.  
	\end{equation}
	In the remaining case $\lambda=0$ one has
	\begin{equation}\label{alpha3}
	\begin{split}
	r_{\alpha,\beta,0}(z) 
	&= \frac{-\alpha + z + 2(z-\delta)\sqrt{\beta(\alpha-z)}}{2z(\alpha-z)} = -\frac{1}{2z} + \frac{\sqrt{\beta}(z-\delta)}{z\sqrt{\alpha-z}}. 
	\end{split}
	\end{equation}
	
\end{prop}

\begin{proof}
By the definition we have $f_{\alpha,\beta,\lambda}(\alpha) = (\lambda\alpha)^2$, substituting this in the expression \eqref{RRR} we obtain that
$
\alpha |\lambda| = 2(\alpha-\delta)\sqrt{\beta (\eta-\alpha)}. 
$
Taking the Taylor expansion around $z=\alpha$ for $\lambda \neq0$ we obtain
\begin{equation}\label{Taylor}
\sqrt{f_{\alpha,\beta,\lambda}(z)}=\alpha|\lambda| + \frac{\sqrt{\beta}(2\eta-3\alpha+\delta)}{\sqrt{\eta-\alpha}}(z-\alpha) + o(|z-\alpha|),\qquad \mbox{as }z\to \alpha. 
\end{equation} 
This implies \eqref{alpha} and \eqref{alpha2}
and so $r_{\alpha,\beta,\lambda}$ may be extended to a continued function on $\C^-\cup \R$. 

The case $\lambda=0$ follows from the fact that in this case we have  $\eta=\alpha$.
\end{proof}
\begin{cor}
	In the case $\lambda<0$ one can extend $r_{\alpha,\beta,\lambda}$ to an analytic function in $\C^-$ and continuous in $\C^- \cup\R$.
\end{cor}

%%%%%%%%%%%%%%%%%%%%%%%%%%
\subsection{Some properties of fGIG distribution}
We study here further properties of free GIG distribution. Some of them motivate Section 4 where we will characterize fGIG distribution in a way analogous to classical GIG distribution. 

The next remark recalls the definition and some basic facts about free Poisson distribution, which will play an important role in this paper.

\begin{rem}\label{rem:freePoisson}
\begin{enumerate}[$1^o$]
\item Marchenko--Pastur (or free-Poisson) distribution $\nu=\nu(\gamma, \lambda)$ is defined by the formula
	\begin{align*}%\label{MPdist}
	\nu=\max\{0,\,1-\lambda\}\,\delta_0+\tilde{\nu},
	\end{align*}
	where $\gamma,\lambda> 0$ and the measure $\tilde{\nu}$, supported on the interval $(\gamma(1-\sqrt{\lambda})^2,\,\gamma(1+\sqrt{\lambda})^2)$, has the density (with respect to the Lebesgue measure)
	$$
	\tilde{\nu}(\dd x)=\frac{1}{2\pi\gamma x}\,\sqrt{4\lambda\gamma^2-(x-\gamma(1+\lambda))^2}\,\dd x. 
	$$
	
\item The $R$-transform of the free Poisson distribution $\nu(\gamma,\lambda)$ is of the form
	\[r_{\nu(\gamma, \lambda)}(z)=\frac{\gamma\lambda}{1-\gamma z}.\]
\end{enumerate}
\end{rem}

The next proposition was proved in \cite[Remark 2.1]{Szp17} which is the free counterpart of a convolution property of classical Gamma and GIG distribution. The proof is a straightforward calculation of the $R$-transform with the help of Remark \ref{rem:F_lambda_sym}. 
\begin{prop}
	\label{GIGPoissConv}
	Let $X$ and $Y$ be free, $X$ free GIG distributed $\mu(\alpha,\beta,-\lambda)$ and $Y$ free Poisson distributed $\nu(1/\alpha,\lambda)$ respectively, for $\alpha,\beta,\lambda>0$. Then $X+Y$ is free GIG distributed $\mu(\alpha,\beta,\lambda)$. 
\end{prop}

We also quote another result from \cite[Remark 2.2]{Szp17} which is again the free analogue of a property of classical GIG distribution. The proof is a simple calculation of the density. 
\begin{prop}\label{GIGInv}
	If $X$ has the free GIG distribution $\mu(\alpha,\beta,\lambda)$ then $X^{-1}$ has the free GIG distribution $\mu(\beta,\alpha,-\lambda)$.
\end{prop}

The two propositions above imply some distributional properties of fGIG distribution. In the Section 4 we will study characterization of the fGIG distribution related to these properties.

\begin{rem}\label{rem:prop}
\begin{enumerate}[$1^o$]
	
\item Fix $\lambda,\alpha>0$. If $X$ has fGIG distribution $\mu(\alpha,\alpha,-\lambda)$ and $Y$ has the free Poisson distribution $\nu(1/\alpha,\lambda)$ and $X,Y$ are free then $X\stackrel{d}{=}(X+Y)^{-1}$.

	Indeed by Proposition \ref{GIGPoissConv} we get that $X+Y$ has fGIG distribution $\mu(\alpha,\alpha,\lambda)$ and now Proposition \ref{GIGInv} implies that $(X+Y)^{-1}$ has the distribution $\mu(\alpha,\alpha,-\lambda)$.
	
\item  One can easily generalize the above observation. Take $\alpha,\beta,\lambda>0$, and $X,Y_1,Y_2$ free, such that $X$ has fGIG distribution $\mu(\alpha,\beta,-\lambda)$, $Y_1$ is free Poisson distributed $\nu(1/\beta,\lambda)$ and $Y_2$ is distributed $\nu(1/\alpha,\lambda)$, then $X\stackrel{d}{=}(Y_1+(Y_2+X)^{-1})^{-1}$.

Similarly as before we have that $X+Y_2$ has distribution $\mu(\alpha,\beta,\lambda)$, then by Proposition \ref{GIGInv} we get that $(X+Y_2)^{-1}$ has distribution $\mu(\beta,\alpha,-\lambda)$. Then we have that $Y_1+(Y_2+X)^{-1}$ has the distribution $\mu(\beta,\alpha,\lambda)$ and finally we get $(Y_1+(Y_2+X)^{-1})^{-1}$ has the desired distribution $\mu(\alpha,\beta,-\lambda)$.
	
\item Both identities above can be iterated finitely many times, so that one obtains that 
	$X\stackrel{d}{=}\left(Y_1+\left(Y_2+\cdots\right)^{-1}\right)^{-1}$, where $Y_1,Y_2,\ldots$ are free, for $k$ odd $Y_k$ has the free Poisson distribution $\nu(1/\beta,\lambda)$  and for $k$ even $Y_k$ has the distribution $\nu(1/\alpha,\lambda)$. For the case described in $1^o$ one simply has to take $\alpha=\beta$. We are not sure if infinite continued fractions can be defined. 
\end{enumerate}
\end{rem}

Next we study limits of the fGIG measure $\mu(\alpha,\beta,\lambda)$ when $\alpha\to 0$ and $\beta\to 0$. This was stated with some mistake in \cite[Remark 2.3]{Szp17}.

\begin{prop}
	As $\beta\downarrow 0$ we have the following weak limits of the fGIG distribution
	\begin{align}\label{eq:limits}
	\lim_{\beta\downarrow 0}\mu(\alpha,\beta,\lambda) = 
	\begin{cases} \nu(1/\alpha, \lambda), & \lambda \geq1, \\ 
	\frac{1-\lambda}{2}\delta_0 + \frac{1+\lambda}{2}\nu(\frac{1+\lambda}{2\alpha},1),  & |\lambda|<1,\\
	\delta_0, & \lambda\leq -1.  
	\end{cases}
	\end{align}
	Taking into account Proposition \ref{GIGInv} one can also describe limits when $\alpha\downarrow 0$ for $\lambda \geq1$. 
\end{prop}
\begin{rem} This result reflects the fact that GIG matrix generalizes the Wishart matrix for $\lambda \geq1$, but not for $\lambda <1$ (see \cite{Fer06} for GIG matrix and \cite{HP00} for the Wishart matrix). 
\end{rem}
\begin{proof}
	We will find the limit by calculating limits of the $R$-transform, since convergence of the $R$-transform implies weak convergence. Observe that from Remark \ref{rem:F_lambda_sym} we can consider only $\lambda\geq0$, however we decided to present all cases, as the consideration will give asymptotic behaviour of support of fGIG measure. In view of \eqref{eq:pol_fpar}, the only non-trivial part is  limits of $\beta\gamma$ when $\beta\to0$.
	Observe that if we define $F(a,b,\alpha,\beta,\lambda)$ by
	\begin{align*}
	\left(1-\lambda+\alpha\sqrt{ab}-\beta\frac{a+b}{2ab},1+\lambda+\frac{\beta}{\sqrt{ab}}-\alpha\frac{a+b}{2}\right)^T &=\left(f(a,b,\alpha,\beta,\lambda),g((a,b,\alpha,\beta,\lambda))\right)^T\\
	&=F(a,b,\alpha,\beta,\lambda). 
	\end{align*} 
	Then the solution to the system \eqref{eq1}, \eqref{eq2} are functions $(a(\alpha,\beta,\lambda),b(\alpha,\beta,\lambda))$, such that $F(a(\alpha,\beta,\lambda),b(\alpha,\beta,\lambda),\alpha,\beta,\lambda)=(0,0)$. To use Implicit Function Theorem, by calculating the Jacobian with respect to $(a,b)$, we observe that $a(\alpha,\beta,\lambda)$ and $b(\alpha,\beta,\lambda)$ are continuous (even differentiable) functions of $\alpha,\beta>0$ and $\lambda\in \mathbb{R}$. 
	
	\textbf{Case 1.} $\lambda>1$\\
	Observe if we take $\beta=0$ then a real solution $0<a<b$ for the system \eqref{eq1}, \eqref{eq2}
	\begin{align}\label{Case1}
	1-\lambda+\alpha\sqrt{ab}&=0\\
	1+\lambda-\alpha\frac{a+b}{2}&=0
	\end{align}
	still exists. 
	Moreover, because at $\beta=0$ Jacobian is non-zero, Implicit Function Theorem says that solutions are continuous at $\beta=0$.
	Thus using \eqref{Case1} we get
	\begin{align*}
	\beta \gamma=\frac{\alpha^2 ab+\tfrac{\beta^2}{ab}-2\alpha\beta (\tfrac{a+b}{\sqrt{ab}}-1)-(\lambda-1)^2}{4}=\frac{\tfrac{\beta^2}{ab}-2\alpha\beta (\tfrac{a+b}{\sqrt{ab}}-1)}{4}.
	\end{align*}
	The above implies that $\beta\gamma\to 0$ when $\beta\to 0$ since $a,b$ have finite and non-zero limit when $\beta\to 0$, as explained above. 
	
	\textbf{Case 2.} $\lambda<-1$\\
	In that case we see that setting $\beta=0$ in \eqref{eq1} leads to an equation with no real solution for $(a,b)$. In this case the part $\beta\tfrac{a+b}{2ab}$ has non-zero limit when $\beta\to 0$. To be precise substitute $a=\beta a^\prime$ and $b=\beta b^\prime$ in \eqref{eq1}, \eqref{eq2}, and then we get
	\begin{align*}
	1-\lambda+\alpha\beta\sqrt{a^\prime b^\prime}-\frac{a^\prime+b^\prime}{2a^\prime b^\prime}=&0\\
	1+\lambda+\frac{1}{\sqrt{a^\prime b^\prime}}-\alpha\beta\frac{a^\prime+b^\prime}{2}=&0.
	\end{align*}
	The above system is equivalent to the system \eqref{eq1},\eqref{eq2} with $\alpha:=\alpha\beta$ and $\beta:=1$.
	If we set $\beta=0$ as in Case 1 we get
	\begin{align}
	1-\lambda-\frac{a^\prime+b^\prime}{2a^\prime b^\prime}=&0\\
	\label{Case2} 1+\lambda+\frac{1}{\sqrt{a^\prime b^\prime}}=&0.
	\end{align}
	The above system has solution $0<a^\prime<b^\prime$ for $\lambda<-1$.  Calculating the Jacobian we see that it is non-zero at $\beta=0$, so Implicit Function Theorem implies that $a^\prime$ and $b^\prime$ are continuous functions at $\beta=0$ in the case $\lambda<-1$. \\
	This implies that in the case $\lambda<-1$ the solutions of \eqref{eq1},\eqref{eq2} are  $a(\beta)= \beta a^\prime+o(\beta)$ and $b(\beta)= \beta b^\prime +o(\beta)$. Thus we have
	\begin{align*}
	\lim_{\beta\to 0}\beta \gamma&=\frac{\alpha^2 ab+\tfrac{\beta^2}{ab}-2\alpha\beta (\tfrac{a+b}{\sqrt{ab}}-1)-(\lambda-1)^2}{4}\\&=\frac{\tfrac{1}{a^\prime b^\prime}-(\lambda-1)^2 }{4}=\frac{(\lambda+1)^2-(\lambda-1)^2 }{4}=\lambda,
	\end{align*}
	where in the equation one before the last we used \eqref{Case2}.
	
	\textbf{Case 3.} $|\lambda|< 1$\\
	Observe that neither \eqref{Case1} nor \eqref{Case2} has a real solution in the case $|\lambda|< 1$. This is because in this case asymptotically $a(\beta)=a^\prime \beta +o(\beta)$ and $b$ has a finite positive limit as $\beta\to0$. Similarly as in Case 2 let us substitute $a=\beta a^\prime $ in \eqref{eq1}, \eqref{eq2}, which gives
	\begin{align*}
	1-\lambda+\alpha\sqrt{\beta a^\prime b}-\frac{\beta a^\prime+b}{2a^\prime b}&=0\\
	1+\lambda+\frac{\sqrt{\beta}}{a^\prime b}-\alpha\frac{\beta a^\prime+b}{2}&=0.
	\end{align*}
If we set $\beta=0$ we get
	\begin{align}
	1-\lambda-\frac{1}{2a^\prime}&=0,\\
	1+\lambda-\alpha\frac{b}{2}&=0,
	\end{align}
	which obviously has positive solution $(a^\prime,b)$ when $|\lambda|<1$. As before the Jacobian is non-zero at $\beta=0$, so $a^\prime$ and $b$ are continuous at $\beta=0$.  \\
	Now we go back to the limit $\lim_{\beta\to 0}\beta\gamma$. We have $a(\beta)=\beta a^\prime+o(\beta)$, thus
	\begin{align*}
	\lim_{\beta\to 0}\beta \gamma=\lim_{\beta\to 0}\frac{\alpha^2 ab+\tfrac{\beta^2}{ab}-2\alpha\beta (\tfrac{a+b}{\sqrt{ab}}-1)-(\lambda-1)^2}{4}=-\frac{(\lambda-1)^2}{4}.
	\end{align*} 

	\textbf{Case 4.} $|\lambda|= 1$\\
	 An analysis similar to the above cases shows that in the case $\lambda=1$ we have $a(\beta)=a^\prime \beta^{2/3}+o(\beta^{2/3})$ and $b$ has positive limit when $\beta\to 0$. In the case $\lambda=-1$ one gets $a(\beta)=a^\prime \beta+o(\beta)$ and $b(\beta)=b^\prime\beta^{1/3}+o(\beta^{1/3})$ as $\beta\to 0$.

Thus we can calculate the limit of $f_{\alpha,\beta,\lambda}$ as $\beta\to 0$ \eqref{eq:pol_fpar},  
\begin{equation}
\lim_{\beta\downarrow0} f_{\alpha,\beta,\lambda}(z) = 
\begin{cases}
(\alpha+(\lambda-1)z)^2, & \lambda>1, \\
 \alpha^2+(\lambda^2-1)\alpha z, & |\lambda| \leq 1, \\
 (\alpha-(\lambda+1)z)^2,  & \lambda<-1. 
 \end{cases}
\end{equation}
The above allows us to calculate limiting $R$-transform and hence the Cauchy transform which implies \eqref{eq:limits}.  
\end{proof}

\begin{cor}
Considering the continuous dependence of roots on parameters shows the following asymptotic behaviour of the double root $\delta<0$ and the simple root $\eta \geq \alpha$. 
\begin{enumerate}[\rm(i)]
\item If $|\lambda|>1$ then 
$\delta\to\alpha/(1-|\lambda|)$ and $\eta\to  +\infty$ as $\beta \downarrow0$. 
\item If $|\lambda|<1$ then 
$\delta\to-\infty$ and $\eta \to \alpha/(1-\lambda^2)$ as $\beta \downarrow0$. 
\item If $\lambda=\pm1$ then 
$\delta\to-\infty$ and $\eta \to +\infty$ as $\beta \downarrow0$. 
\end{enumerate}
\end{cor}

%%%%%%%%%%%%%%%%%%%%%%%%%%%%%%%
%%%%%%%%%%%%%%%%%%%%%%%%%%%%%%%%
\section{Regularity of fGIG distribution under free convolution} 
In this section we study in detail regularity properties of the fGIG distribution related to the operation of free additive convolution. In the next theorem we collect all the results proved in this section. The theorem contains several statements about free GIG distributions. Each subsection of the present section proves a part of the theorem.
\begin{thm}\label{thm:sec3}
	The following holds for the free GIG measure $\mu(\alpha,\beta,\lambda)$:
	\begin{enumerate}[$1^o$]
		\item It is freely infinitely divisible for any $\alpha,\beta>0$ and $\lambda\in\mathbb{R}.$
		\item The free Levy measure is of the form
		\begin{equation}\label{FLM}
		\tau_{\alpha,\beta,\lambda}(\dd x)=\max\{\lambda,0\} \delta_{1/\alpha}(\dd x) + \frac{(1-\delta x) \sqrt{\beta (1-\eta x)}}{\pi x^{3/2} (1-\alpha x)} 1_{(0,1/\eta)}(x)\, \dd x.
		\end{equation}
		\item It is free regular with zero drift for all $\alpha,\beta>0$ and $\lambda\in\mathbb{R}$.
		\item It is freely self-decomposable for $\lambda \leq -\frac{B^{\frac{3}{2}}}{A\sqrt{9B-8A}}.$
		\item It is unimodal.
	\end{enumerate}
\end{thm}

%%%%%%%%%%%%%%%%%%%%%%%%%%%%%%%%
\subsection{Free infinite divisibility and free L\'evy measure}

As we mentioned before, having the operation of free convolution defined, it is natural to study infinite divisibility with respect to $\boxplus$. We say that $\mu$ is freely infinitely divisible if for any $n\geq 1$ there exists a probability measure $\mu_n$ such that \[\mu=\underbrace{\mu_n\boxplus\ldots\boxplus\mu_n}_{\text{$n$ times}}.\]
It turns out that free infinite divisibility of compactly supported measures can by described in terms of analytic properties of the $R$-transform. In particular it was proved in \cite[Theorem 4.3]{Voi86} that the free infinite divisibility is equivalent to the inequality $\im(r_{\alpha,\beta,\lambda}(z)) \leq0$ for all $z\in\C^-$.

As in the classical case, for freely infinitely divisible probability measures, one can represent its free cumulant transform with a L\'evy--Khintchine type formula.
For a probability measure $\mu$ on $\R$, the \emph{free cumulant transform} is defined by
\begin{equation}
\FC_\mu(z) = z r_\mu(z). 
\end{equation}
Then $\mu$ is FID if and only if $\FC_\mu$ can be analytically extended to $\C^-$ via the formula 
\begin{equation}\label{FLK2}
\FC_{\mu}(z)=\xi z+\zeta z^2+\int_{\R}\left( \frac{1}{1-z x}-1-z x \,1_{[-1,1] }(x) \right) \tau(\dd x) ,\qquad z\in \C^-, 
\end{equation}
where $\xi \in \R,$ $\zeta\geq 0$ and $\tau$ is a measure on $\R$ such that 
\begin{equation}
\tau (\{0\})=0,\qquad \int_{\R}\min \{1,x^2\}\tau (\dd x) <\infty.
\end{equation}
The triplet $(\xi,\zeta,\tau)$ is called the \emph{free characteristic triplet} of $\mu$, and $\tau$ is called the \emph{free L\'evy measure} of $\mu$. The formula \eqref{FLK2} is called the \emph{free L\'evy--Khintchine formula}.  

\begin{rem}
	The above form of free L\'evy--Khintchine formula was obtained by Barndorff-Nielsen and Thorbj{\o}rnsen \cite{BNT02b} and it has a probabilistic interpretation (see \cite{Sat13}). Another form was obtained by Bercovici and Voiculescu \cite{BV93}, which is more suitable for limit theorems. 
\end{rem}

In order to prove that all fGIG distributions are freely infinitely divisible we will use the following lemma.

\begin{lem}\label{lem:32}
	Let $f\colon (\mathbb C^{-} \cup \mathbb R) \setminus \{x_0\} \to \mathbb C$ be a continuous function, where $x_0\in\mathbb{R}$. Suppose that $f$ is analytic in $\mathbb{C}^{-}$, $f(z)\to 0$ uniformly with $z\to \infty$ and $\im (f(x))\leq 0$ for $x \in \mathbb R \setminus \{x_0\}$. Suppose moreover that $\im(f(z))\leq 0$ for $\im(z)\leq 0$ in neighbourhood of $x_0$ then $\im (f(z))\leq 0$ for all $z\in \mathbb{C}^{-}$. 
\end{lem}
\begin{proof}
	Since $f$ is analytic the function $\im f$ is harmonic and thus satisfies the maximum principle. Fix $\ep>0$. Since $f(z)\to 0$ uniformly with $z\to \infty,$ let $R>0$ be such that $\im f(z)<\ep$. Consider a domain $D_\ep$ with the boundary 
	\[ \partial D_\ep=[-R,x_0-\ep] \cup \{x_0+\ep e^{\ii \theta}: \theta \in[-\pi,0]\} \cup[x_0+\ep,R]\cup\{R e^{\ii \theta}: \theta \in[-\pi,0]\}\] 
	Observe that on $\partial D_\ep$ $\im f(z)<\ep$ by assumptions, and hence by the maximum principle we have $\im f(z)<\ep$ on whole $D_\ep$. Letting $\ep\to 0$ we get that $\im f(z) \leq 0$ on $\mathbb{C}^{-}$. 
\end{proof}

Next we proceed with the proof of free infinite divisibility of fGIG distributions.
\begin{proof}[Proof of Theorem \ref{thm:sec3} $1^o$]
$ $

{\bf Case 1.} $\lambda<1$.

Observe that we have
\begin{equation}\label{hypo}
\im(r_{\alpha,\beta,\lambda}(x)) \leq 0, \qquad x\in\R\setminus\{\alpha\}. 
\end{equation}
From \eqref{R} we see that $\im(r_{\alpha,\beta,\lambda}(x))=0$ for $x \in(-\infty,\alpha) \cup (\alpha,\eta]$, and 
\begin{equation}\label{r}
\im(r_{\alpha,\beta,\lambda}(x)) = \frac{(x-\delta)\sqrt{\beta (x-\eta)}}{x (\alpha-x)}<0,\qquad x>\eta
\end{equation}
since $\eta>\alpha>0>\delta$.\\
Moreover observe that by \eqref{alpha} for $\ep>0$ small enough we have $\im(r_{\alpha,\beta,\lambda}(\alpha+\ep e^{i \theta}))<0$, for $\theta\in[-\pi,0]$. 
Now Lemma \ref{lem:32} implies that free GIG distribution if freely ID in the case $\lambda>0$. 

{\bf Case 2} $\lambda>0$.

In this case similar argument shows that $\mu(\alpha,\beta,\lambda)$ is FID. Moreover by $\eqref{alpha2}$ point $z=\alpha$ is a removable singularity and $r_{\alpha,\beta,\lambda}$ extends to a continuous function on $\C^-\cup\R$. Thus one does not need to take care of the behaviour around $z=\alpha$.

{\bf Case 3} $\lambda=0$.
  
For $\lambda=0$ one can adopt a similar argumentation using \eqref{alpha3}. It also follows from the fact that free GIG family $\mu(\alpha,\beta,\lambda)$ is weakly continuous with respect to $\lambda$. Since free infinite divisibility is preserved by weak limits, then the case $\lambda=0$ may be deduced from the previous two cases. 
\end{proof}

Next we will determine the free L\'evy measure of free GIG distribution $\mu(\alpha,\beta,\lambda)$. 
\begin{proof}[Proof of Theorem \ref{thm:sec3} $2^o$]
Let $(\xi_{\alpha,\beta,\lambda}, \zeta_{\alpha,\beta,\lambda},\tau_{\alpha,\beta,\lambda})$ be the free characteristic triplet of the free GIG distribution $\mu(\alpha,\beta,\lambda)$. By the Stieltjes inversion formula mentioned in Remark \ref{rem:Cauchy}, the absolutely continuous part of the free L\'evy measure has the density 
\begin{equation}
-\lim_{\ep\to0}\frac{1}{\pi x^2}\im(r_{\alpha,\beta,\lambda}(x^{-1}+\ii\ep)), \qquad x \neq 0, 
\end{equation}
atoms are at points $1/p~(p\neq0)$, such that the weight is given by  
\begin{equation}
\tau_{\alpha,\beta,\lambda}(\{1/p\})=\lim_{z\to p} (p - z) r_{\alpha,\beta,\lambda}(z),  
\end{equation}
is non-zero, where $z$ tends to $p$ non-tangentially from $\C^-$. In our case the free L\'evy measure does not have a singular continuous part since $r_{\alpha,\beta,\lambda}$ is continuous on $\C^-\cup\R \setminus\{\alpha\}$.  
Considering \eqref{alpha}--\eqref{alpha3} and \eqref{r} we obtain the free L\'evy measure 
\begin{equation}
\tau_{\alpha,\beta,\lambda}(\dd x)=\max\{\lambda,0\} \delta_{1/\alpha}(\dd x) + \frac{(1-\delta x) \sqrt{\beta (1-\eta x)}}{\pi x^{3/2} (1-\alpha x)} 1_{(0,1/\eta)}(x)\, \dd x.
\end{equation}
Recall that $\eta \geq \alpha >0>\delta$ holds, and $\eta=\alpha$ if and only if $\lambda=0$. 
The other two parameters $\xi_{\alpha,\beta,\lambda}$ and $\zeta_{\alpha,\beta,\lambda}$ in the free characteristic triplet will determined in Section \ref{sec:FR}. 
\end{proof}

%%%%%%%%%%%%%%%%%%%%%%%
\subsection{Free regularity}\label{sec:FR}
In this subsection we will deal with a property stronger than free infinite divisibility, so called free regularity.\\
Let $\mu$ be a FID distribution with the free characteristic triplet $(\xi,\zeta,\tau)$. 
When the semicircular part $\zeta$ is zero and the free L\'evy measure $\tau$ satisfies a stronger integrability property $\int_{\R}\min\{1,|x|\}\tau(\dd x) < \infty$,  then the free L\'evy-Khintchine representation reduces to 
\begin{equation} \label{FR}
\FC_\mu(z)=\xi' z+\int_{\R}\left( \frac{1}{1-z x}-1\right) \tau(\dd x) ,\qquad z\in \C^-,
\end{equation} 
where $\xi' =\xi -\int_{[-1,1]}x \,\tau(\dd x) \in \R$ is called a \emph{drift}. The distribution $\mu$ is said to be \emph{free regular} \cite{PAS12} if $\xi'\geq0$ and $\tau$ is supported on $(0,\infty)$. A probability measure $\mu$ on $\R$ is free regular if and only if the free convolution power $\mu^{\boxplus t}$ is supported on $[0,\infty)$ for every $t>0$, see \cite{AHS13}. Examples of free regular distributions include positive free stable distributions, free Poisson distributions and powers of free Poisson distributions \cite{Has16}. A general criterion in \cite[Theorem 4.6]{AHS13} shows that some boolean stable distributions \cite{AH14} and many probability distributions \cite{AHS13,Has14,AH16} are free regular. A recent result of Ejsmont and Lehner \cite[Proposition 4.13]{EL} provides a wide class of examples: given a nonnegative definite complex matrix $\{a_{ij}\}_{i,j=1}^n$ and free selfadjoint elements $X_1,\dots, X_n$ which have symmetric FID distributions, the polynomial $\sum_{i,j=1}^n a_{ij} X_i X_j$ has a free regular distribution with zero drift.

\begin{proof}[Proof of Theorem \ref{thm:sec3} $3^o$]
For the free GIG distributions, the semicircular part can be found by $\displaystyle\zeta_{\alpha,\beta,\lambda}= \lim_{z\to \infty} z^{-1} r_{\alpha,\beta,\lambda}(z)=0$.  
The free L\'evy measure \eqref{FLM} satisfies
\begin{equation}
\supp(\tau_{\alpha,\beta,\lambda}) \subset (0,\infty), \qquad \int_0^\infty \min\{1,x\} \tau_{\alpha,\beta,\lambda}(\dd x)<\infty
\end{equation}
and so we have the reduced formula \eqref{FR}. The drift is given by $\displaystyle \xi_{\alpha,\beta,\lambda}'=\lim_{u \to -\infty} r_{\alpha,\beta,\lambda}(u)=0$.
\end{proof}

%%%%%%%%%%%%%%%%%%%%%%%
\subsection{Free selfdecomposability}
Classical GIG distribution is selfdecomposable \cite{Hal79,SS79} (more strongly, hyperbolically completely monotone \cite[p.\ 74]{Bon92}), and hence it is natural to ask whether free GIG distribution is freely selfdecomposable. \\
A distribution $\mu$ is said to be \emph{freely selfdecomposable} (FSD) \cite{BNT02a} if  for any $c\in(0,1)$ there exists a probability measure $\mu_c$ such that $\mu= (D_c\mu) \boxplus \mu_c $, where $D_c\mu$ is the dilation of $\mu$, namely $(D_c\mu)(B)=\mu(c^{-1}B)$ for Borel sets $B \subset \R$. A distribution is FSD if and only if it is FID and its free L\'evy measure is of the form 
\begin{equation}\label{SD Levy}
\frac{k(x)}{|x|}\, \dd x, 
\end{equation}
where $k\colon \R \to [0,\infty)$ is non-decreasing on $(-\infty,0)$ and non-increasing on $(0,\infty)$. 
Unlike the free regular distributions, there are only a few known examples of FSD distributions: the free stable distributions, some free Meixner distributions, the classical normal distributions and a few other distributions (see \cite[Example 1.2, Corollary 3.4]{HST}). The free Poisson distribution is not FSD.  
\begin{proof}[Proof of Theorem \ref{thm:sec3} $4^o$]
 In view of \eqref{FLM}, the free GIG distribution $\mu(\alpha,\beta,\lambda)$ is not FSD if $\lambda > 0$. Suppose $\lambda\leq0$, then $\mu(\alpha,\beta,\lambda)$ is FSD if and only if the function 
\begin{equation}
k_{\alpha,\beta,\lambda}(x)=\frac{(1-\delta x) \sqrt{\beta (1-\eta x)}}{\pi \sqrt{x} (1-\alpha x)} 
\end{equation}
is non-increasing on $(0,1/\eta)$. The derivative is 
\begin{equation}
k_{\alpha,\beta,\lambda}'(x) = - \frac{\sqrt{\beta} [1+(\delta-3\alpha)x +(2 \alpha \eta -2\eta\delta + \alpha \delta )x^2]}{2\pi x^{3/2} (1-\alpha x)^2 \sqrt{1-\eta x}}. 
\end{equation}
Hence FSD is equivalent to 
\begin{equation}
g(x):=1+(\delta-3\alpha)x +(2 \alpha \eta -2\eta\delta + \alpha \delta )x^2  \geq 0,\qquad 0\leq x \leq 1/\eta. 
\end{equation}
Using $\eta \geq\alpha >0>\delta$, one can show that $2 \alpha \eta -2\eta\delta + \alpha \delta >0$, a straightforward calculation shows that the function $g$ takes a minimum at a point in $(0,1/\eta)$. Thus FSD is equivalent to 
\begin{equation}
D:=(\delta-3\alpha)^2 - 4 (2 \alpha \eta -2\eta\delta + \alpha \delta ) \leq 0. 
\end{equation}
In order to determine when the above inequality holds, it is convenient to switch to parameters $A,B$ defined by \eqref{eq:AB}. 
Using formulas derived in Section \ref{sec:form} we obtain
\begin{equation}
D= \frac{4(B+\lambda A)(8\lambda^2 A^3 -9\lambda^2 A^2 B +B^3)}{A^2  B (A-B)^2 (B-\lambda A)}. 
\end{equation}
Calculating $\lambda$ for which $D$ is non-positive we obtain that \[\lambda\leq-\frac{B^{\frac{3}{2}}}{A\sqrt{9B-8A}}.\]
\end{proof}
\begin{cor} One can easily find that the maximum of the function $-\frac{B^{\frac{3}{2}}}{A\sqrt{9B-8A}}$ over $A,B\geq 0$ equals  $-\frac{4}{9}\sqrt{3}$.
Thus the set of parameters $(A,B)$ that give FSD distributions is nonempty if and only if $\lambda \leq -\frac{4}{9}\sqrt{3}$.

In the critical case $\lambda = -\frac{4}{9}\sqrt{3}$ only the pairs $(A, \frac{4}{3}A), A>0$ give FSD distributions. If one puts $A=12 t, B= 16 t$ then $a=(2-\sqrt{3})^2t,b=(2+\sqrt{3})^2t$, $\alpha = \frac{3-\sqrt{3}}{18 t}$, $\beta = \frac{3+\sqrt{3}}{18}t, \delta= -\frac{3-\sqrt{3}}{6t}=- 2\eta$. 
One can easily show that $\mu(\alpha,\beta,-1)$ is FSD if and only if $(0<A<)~B\leq \frac{-1+\sqrt{33}}{2} A$. 

Finally note that the above result is in contrast to the fact that classical GIG distributions are all selfdecomposable. 
\end{cor}

%%%%%%%%%%%%%%%%%%%
\subsection{Unimodality}
Since relations of free infinite divisibility and free self decomposability were studied in the literature, we decided to determine whether measures from the free GIG family are unimodal.

A measure $\mu$ is said to be \emph{unimodal} if for some $c\in\R$
\begin{equation}\label{UM}
\mu(\dd x)=\mu(\{c\})\delta_c(\dd x)+f(x)\, \dd x, 
\end{equation}
where $f\colon\R\to[0,\infty)$ is non-decreasing on $(-\infty,c)$ and non-increasing on $(c,\infty)$. In this case $c$ is called the \emph{mode}. 
Hasebe and Thorbj{\o}rnsen \cite{HT16} proved that FSD distributions are unimodal. Since some free GIG distributions are not FSD, the result from \cite{HT16} does not apply. However it turns out that free GIG measures are unimodal.
\begin{proof}[Proof of Theorem \ref{thm:sec3} $5^o$]
Calculating the derivative of the density of $\mu(\alpha,\beta,\lambda)$ one obtains
\begin{equation}
\frac{ x (a + b - 2 x)(x \alpha + \tfrac{\beta}{\sqrt{a b}}) - 
2 (b - x) (x-a) (x \alpha + \tfrac{2 \beta}{\sqrt{
a b}})}{2 x^3 \sqrt{(b - x) (x-a)}}
\end{equation}
Denoting by $f(x)$ the quadratic polynomial in the numerator, one can easily see from the shape of the density that $f(a)>0>f(b)$ and hence the derivative vanishes at a unique point in $(a,b)$ (since $f$ is quadratic).
\end{proof}

%%%%%%%%%%%%%%%%%%%%%%%%%%%%
%%%%%%%%%%%%%%%%%%%%%%%%%%%%
\section{Characterizations the free GIG distribution}

In this section we show that the fGIG distribution can be characterized similarly as classical GIG distribution. In \cite{Szp17} fGIG was characterized in terms of free independence property, the classical probability analogue of this result characterizes classical GIG distribution. In this section we find two more instances where such analogy holds true, one is a characterization by some distributional properties related with continued fractions, the other is maximization of free entropy.

%%%%%%%%%%%%%%%%%%%%%%%%%%
\subsection{Continued fraction characterization}
In this section we study a characterization of fGIG distribution which is analogous to the characterization of GIG distribution proved in \cite{LS83}. Our strategy is different from the one used in \cite{LS83}. We will not deal with continued fractions, but we will take advantage of subordination for free convolutions, which allows us to prove the simpler version of ''continued fraction'' characterization of fGIG distribution.

\begin{thm}\label{thm:char}
Let $Y$ have the free Poisson distribution $\nu(1/\alpha,\lambda)$ and let $X$ be free from $Y$, where $\alpha,\lambda>0$ and $X>0$, then we have
\begin{align}\label{eq:distr_char}
X\stackrel{d}{=}\left(X+Y\right)^{-1}
\end{align}
if and only if $X$ has free GIG distribution $\mu(\alpha,\alpha,-\lambda)$.
\end{thm}

\begin{rem}
	Observe that the "if" part of the above theorems is contained in the remark \ref{rem:prop}. We only have to show that if $\eqref{eq:distr_char}$ holds where $Y$ has free Poisson distribution $\nu(1/\alpha,\lambda)$, then $X$ has free GIG distribution.
\end{rem}

As mentioned above our proof of the above theorem uses subordination of free convolution. This property of free convolution was first observed by Voiculescu \cite{Voi93} and then generalized by Biane \cite{Bia98}. Let us shortly recall what we mean by subordination of free additive convolution.

\begin{rem}
Subordination of free convolution states that for probability measures
$\mu,\nu$, there exists an analytic function defined on $\mathbb{C}\setminus\mathbb{R}$ with the property $F(\overline{z})=\overline{F(z)}$ such that for  $z\in\mathbb{C}^+$ we have $\im F(z)>\im z$ and
\[G_{\mu\boxplus\nu}(z)=G_\mu(\omega(z)).\]

Now if we denote by $\omega_1$ and $\omega_2$ subordination functions such that $G_{\mu\boxplus\nu}=G_\mu(\omega_1)$ and $G_{\mu\boxplus\nu}=G_\nu(\omega_2)$, then $\omega_1(z)+\omega_2(z)=1/G_{\mu\boxplus\nu}(z)+z$.
\end{rem}

Next we proceed with the proof of Theorem \ref{thm:char} which is the main result of this section.

\begin{proof}[Proof of Theorem \ref{thm:char}]
	First note that \eqref{eq:distr_char} is equivalent to 
	\[\frac{1}{X}\stackrel{d}{=}X+Y,\]
	Which may be equivalently stated in terms of Cauchy transforms of both sides as
	\begin{align}\label{eqn:CharCauch}
	G_{X^{-1}}(z)=G_{X+Y}(z).
	\end{align}
	Subordination allows as to write the Cauchy transform of $X+Y$ in two ways
	\begin{align}
	\label{Sub1}
	G_{X+Y}(z)&=G_X(\omega_X(z)),\\
	\label{Sub2}
	G_{X+Y}(z)&=G_Y(\omega_Y(z)).
	\end{align}
	Moreover $\omega_X$ and $\omega_Y$ satisfy
	\begin{align*}
	\omega_X(z)+\omega_Y(z)=1/G_{X+Y}(z)+z.
	\end{align*} 
	From the above we get 
	\begin{align}\label{subrel}
	\omega_X(z)=1/G_{X+Y}(z)+z-\omega_Y(z),
	\end{align} 
	this together with \eqref{eqn:CharCauch} and \eqref{Sub1} gives
	\begin{align*}
	G_{X^{-1}}(z)&=G_X\left(\frac{1}{G_{X^{-1}}}(z)+z-\omega_Y(z)\right).
	\end{align*}
	Since we know that $Y$ has free Poisson distribution $\nu(\lambda,1/\alpha)$ we can calculate $\omega_Y$ in terms of $G_{X^{-1}}$ using \eqref{Sub2}. To do this one has to use the identity $G_Z^{\langle -1\rangle}(z)=r_Z(z)+1/z$ for any self-adjoint random variable $Z$ and the form of the $R$-transform of free Poisson distribution recalled in Remark \ref{rem:freePoisson}. 
	\begin{align}\label{omegax}
	\omega_Y(z)=\frac{\lambda }{\alpha-G_{X^{-1}}(z) }+\frac{1}{G_{X^{-1}}(z)}
	\end{align}
	Now we can use \eqref{Sub1}, where we substitute $G_{X+Y}(z)=G_{X^{-1}}(z)$ to obtain 
	\begin{align}\label{FE}	G_{X^{-1}}(z)=G_{X}\left(\frac{\lambda }{G_{X^{-1}}(z)-\alpha }+z\right).
	\end{align}
	Next we observe that we have 
	\begin{align}\label{CauchInv}
	G_{X^{-1}}(z)=\frac{1}{z}\left(-\frac{1}{z}G_X\left(\frac{1}{z}\right)+1\right),
	\end{align} 
	which allows to transform \eqref{FE} to an equation for $G_X$. It is enough to show that this equation has a unique solution. Indeed from Remark \ref{rem:prop} we know that free GIG distribution $\mu(\alpha,\alpha,\lambda)$ has the desired property, which in particular means that for $X$ distributed $\mu(\alpha,\alpha,\lambda)$ equation \eqref{FE} is satisfied. Thus if there is a unique solution it has to be the Cauchy transform of the free GIG distribution.

	To prove uniqueness of the Cauchy transform of $X$, we will prove that coefficients of the expansion of $G_X$ at a special ``good'' point, are uniquely determined by $\alpha$ and $\lambda$.

	First we will determine the point at which we will expand the function. Observe that with our assumptions $G_{X^{-1}}$ is well defined on the negative half-line, moreover $G_{X^{-1}}(x)<0$ for any $x<0$, and we have $G_{X^{-1}}(x)\to0$ with $x\to-\infty$.
On the other hand the function $f(x)=1/x-x$ is decreasing on the negative half-line, and negative for $x\in(-1,0)$. Thus there exist a unique point $c\in(-1,0)$ such that
\begin{equation}\label{key eq}
\frac{1}{c} = \frac{\lambda}{G_{X^{-1}}(c)-\alpha}+c. 
\end{equation}
Let us denote
\begin{equation*}
M(z):= G_X\left(\frac{1}{z}\right) 
\end{equation*}
and 
\begin{equation}\label{eqn:funcN}
N(z):=\left(\frac{\lambda}{G_{X^{-1}}(z)-\alpha}+z\right)^{-1} = \frac{-z+ \alpha z^2 +M(z)}{-(1+\lambda)z^2 +\alpha z^3 + z M(z)},
\end{equation}
where the last equality follows from \eqref{CauchInv}.\\
One has $N(c)=c$, and our functional equation \eqref{FE} may be rewritten (with the help of \eqref{CauchInv}) as
\begin{equation}\label{FE2}
-M(z) +z = z^2 M(N(z)). 
\end{equation}
Functions $M$ and $N$ are analytic around any $x<0$. Consider the expansions
\begin{align*}
M(z) &= \sum_{n=0}^\infty \alpha_n (z-c)^n, \\
N(z) &= \sum_{n=0}^\infty \beta_n (z-c)^n.  
\end{align*}
Observe that $\beta_0=c$ since $N(c)=c$. Differentiating \eqref{eqn:funcN} we observe that any $\beta_n,\, n\geq1$ is a rational function of $\alpha, \lambda, c, \alpha_0,\alpha_1,\dots, \alpha_n$. Moreover any $\beta_n,\,n\geq 1 $ is a degree one polynomial in $\alpha_n$. We have
\begin{align}
\beta_n = \frac{-\lambda}{[\alpha_0 -(1+\lambda)c+\alpha c^2]^2} \alpha_n + R_n, 
\end{align}
where $R_n$ is a rational function of $n+3$ variables evaluated at $(\alpha,\lambda,c,\alpha_0,\alpha_1,\dots, \alpha_{n-1})$, which does not depend on the distribution of $X$. For example $\beta_1$ is given by 
\begin{equation}\label{eq:beta1}
\begin{split}
\beta_1
&=N'(c) =\left. \left(\frac{-z+ \alpha z^2 +M(z)}{-(1+\lambda)z^2 +\alpha z^3 + z M(z)}\right)' \right|_{z=c} \\
&=\frac{-\lambda c^2 \alpha_1 + c^2(-1-\lambda +2\alpha c -\alpha^2 c^2)+2c(1+\lambda-\alpha c)\alpha_0- \alpha_0^2 }{c^2[\alpha_0-(1+\lambda)c+\alpha c^2]^2}. 
\end{split}
\end{equation}

Next we investigate some properties of $c, \alpha_0$ and $\alpha_1$. Evaluating both sides of \eqref{FE2} at $z=c$ yields 
\begin{equation*}
-M(c)+c = c^2 M(N(c)) = c^2M(c), 
\end{equation*} 
since $M(c)=\alpha_0$ we get
\begin{equation}\label{eq1c}
\alpha_0=\frac{c}{1+c^2}. 
\end{equation}
Observe that $\alpha_0= M(c) = G_X(1/c)$ and $\alpha_1=M'(c)=-c^{-2} G_X'(1/c)$ hence we have  
\begin{equation*}
\frac{1}{1+c^2}= \int_{0}^\infty \frac{1}{1-c x} \dd\mu_X( x), \qquad \alpha_1 =  \int_{0}^\infty \frac{1}{(1-c x)^2} \dd\mu_X( x),
\end{equation*}
where $\mu_X$ is the distribution of $X$. Using the Schwarz inequality for the first estimate and a simple observation that $0\leq1/(1-cx) \leq1$ for $x>0$, for the latter estimate we obtain
\begin{equation}\label{eq:alpha1}
\frac{1}{(1+c^2)^{2}}=\left(\int_{0}^\infty \frac{1}{1-c x} \mu_X(\dd x)\right)^2 \leq \int_{0}^\infty \frac{1}{(1-c x)^2} \mu_X(\dd x)= \alpha_1 \leq \frac{1}{1+c^2}. 
\end{equation}
%{\color{red} Note: The lower bound is not used so we may omit it.}
The equation \eqref{key eq} together with \eqref{CauchInv} gives 
\begin{equation}\label{eq2c}
\frac{1}{c} = \frac{\lambda c^2}{-\alpha_0 + c - \alpha c^2} +c. 
\end{equation}
Substituting \eqref{eq1c} to \eqref{eq2c} after simple calculations we get
\begin{equation}\label{eq:c}
\alpha c^4 - (1+\lambda)c^3+(1-\lambda)c -\alpha = 0. 
\end{equation} 

We start by showing that $\alpha_0$ is determined only by $\alpha$ and $\lambda$. We will show that $c$, which we showed before is a unique number, depends only on $\alpha$ and $\lambda$ and thus \eqref{eq1c} shows that $\alpha_0$ is determined by $\alpha$ and $\lambda$.

Since the polynomial $c^4 - (1+\lambda)c^3$ is non-negative for $c<0$ and has a root at $c=0$, and the polynomial $(\lambda-1)c +\alpha$ equals $\alpha>0$ at $c=0$ it follows that there is only one negative $c$, such that the two polynomials are equal and thus the number $c$ is uniquely determined by $(\alpha,\lambda)$. From \eqref{eq1c} we see that $\alpha_0$ is also uniquely determined by $(\alpha,\lambda)$.

Next we will prove that $\alpha_1$ only depends on $\alpha$ and $\lambda$. 
Differentiating \eqref{FE2} and evaluating at $z=c$ we obtain
\begin{equation}\label{eq3c}
1-\alpha_1 = 2 c \alpha_0 + c^2 \alpha_1 \beta_1.  
\end{equation}
Substituting $\alpha_0$ and $\lambda$ from the equations \eqref{eq1c} and \eqref{eq:c} we simplify \eqref{eq:beta1} and we get
\begin{equation*}
\beta_1 = \frac{(1-c^4)\alpha_1 -1+2c^2-\alpha c^3 -\alpha c^5}{c(\alpha-c+\alpha c^2)}
\end{equation*}
and then equation \eqref{eq3c} may be expressed in the form 
\begin{equation}\label{eq:alpha2}
 c(1+c^2)^2 \alpha_1^2 + (\alpha(1+c^2)^2-2c )(1+c^2)\alpha_1 -(\alpha-c + \alpha c^2) =0.  
\end{equation}
The above is a degree 2 polynomial in $\alpha_1$, denote this polynomial by $f$, we have then 
\begin{equation*}
f(0) <0,\qquad f\left(\frac{1}{1+c^2}\right)  = \alpha c^2 (1+c^2)>0. 
\end{equation*}
Where the first inequality follows from the fact that $c<0$.
Since the coefficient $c(1+c^2)^2$ is negative we conclude that $f$ has one root in the interval $(0,1/(1+c^2))$ and the other in $(1/(1+c^2),\infty)$. The inequality \eqref{eq:alpha1} implies that $\alpha_1$ is the smaller root of $f$, which is a function of $ \alpha$ and $c$ and hence of $\alpha$ and $\lambda$.

In order to prove that $\alpha_n$ depends only on $(\alpha,\lambda)$ for $n\geq2$, first we estimate  $\beta_1$. Note that \eqref{eq3c} and \eqref{eq1c} imply that 
\begin{equation}\label{eq4c}
\beta_1 = \frac{1-c^2}{\alpha_1 c^2(1+c^2)} -\frac{1}{c^2}.
\end{equation} 
Combining this with the inequality \eqref{eq:alpha1} we easily get that
\begin{equation}\label{eq:beta}
-1 \leq \beta_1 \leq -c^2. 
\end{equation}
%{\color{red} Note: The upper bound is not used so we may omit it.}

Now we prove by induction on $n$ that $\alpha_n$ only depends on $\alpha$ and $\lambda$.  For $n\geq2$ differentiating $n$-times \eqref{FE2} and evaluating at $z=c$ we arrive at
\begin{equation}\label{eq5c}
-\alpha_n = c^2(\alpha_n \beta_1^n + \alpha_1 \beta_n) + Q_n, 
\end{equation}
where $Q_n$ is a universal polynomial (which means that the polynomial does not depend on the distribution of $X$) in $2n+1$ variables evaluated at $(\alpha,\lambda,c,\alpha_1,\dots, \alpha_{n-1}, \beta_1,\cdots, \beta_{n-1})$. According to the inductive hypothesis, the polynomials $R_n$ and $Q_n$ depend only on $\alpha$ and $\lambda$.  We also have that $\beta_n=p \alpha_n + R_n$, where 
\begin{equation*}
p :=  \frac{-\lambda}{[\alpha_0 -(1+\lambda)c+\alpha c^2]^2} = \frac{1-c^4}{c(\alpha-c+\alpha c^2)}. 
\end{equation*}
The last formula is obtained by substituting $\alpha_0$ and $\lambda$ from \eqref{eq1c} and \eqref{eq:c}. The equation \eqref{eq5c} then becomes
\begin{equation*}
(1+c^2\beta_1^n + c^2 p \alpha_1)\alpha_n + c^2 \alpha_1 R_n + Q_n=0. 
\end{equation*}
The inequalities \eqref{eq:alpha1} and \eqref{eq:beta} show that 
\begin{equation*}
\begin{split}
1+c^2\beta_1^n + c^2 p \alpha_1 
&\geq 1-c^2 +\frac{c^2(1-c^4)}{c(\alpha-c+\alpha c^2) (1+c^2)} =\frac{\alpha(1-c^4)}{\alpha-c+\alpha c^2}>0, 
\end{split}
\end{equation*}
thus $1+c^2\beta_1^n + c^2 p \alpha_1$ is non-zero. Therefore, the number $\alpha_n$ is uniquely determined by $\alpha$ and $\lambda$. 

Thus we have shown that, if a random variable $X>0$ satisfies the functional equation \eqref{FE} for fixed $\alpha>0$ and $\lambda>0$, then the point $c$ and all the coefficients $\alpha_0,\alpha_1,\alpha_2,\dots$ of the series expansion of $M(z)$ at $z=c$ are determined only by $\alpha$ and $\lambda$. By analytic continuation, the Cauchy transform $G_X$ is determined uniquely by $\alpha$ and $\lambda$, so there is only one distribution of $X$ for which this equation is satisfied.
\end{proof}

%%%%%%%%%%%%%%%%%%%%%%%%%%%
\subsection{Remarks on free entropy characterization} 
F\'eral \cite{Fer06} proved that fGIG $\mu(\alpha,\beta,\lambda)$ is a unique probability measure which maximizes the following free entropy functional with potential
\begin{align*}
I_{\alpha,\beta,\lambda}(\mu)=\int\!\!\!\int \log|x-y|\, \dd\mu(x) \dd\mu(y)-\int V_{\alpha,\beta,\lambda}(x)\, \dd\mu(x),   
\end{align*} 
among all the compactly supported probability measures $\mu$ on $(0,\infty)$, where $\alpha, \beta>0$ and $\lambda \in \R$ are fixed constants, and 
$$
V_{\alpha,\beta,\lambda}(x)=(1-\lambda) \log x+\alpha x+\frac{\beta}{x}. 
$$ 

Here we point out the classical analogue. The (classical) GIG distribution is the probability measure on $(0,\infty)$ with the density 
\begin{equation}\label{C-GIG}
\frac{(\alpha/\beta)^{\lambda/2}}{2K_\lambda(2\sqrt{\alpha\beta})} x^{\lambda-1} e^{-(\alpha x + \beta /x)}, \qquad \alpha,\beta>0, \lambda\in\R, 
\end{equation}
where $K_\lambda$ is the modified Bessel function of the second kind. Note that this density is proportional to $\exp(-V_{\alpha,\beta,\lambda}(x))$. 
Kawamura and Iwase \cite{KI03} proved that the GIG distribution is a unique probability measure which maximizes the classical entropy with the same potential 
$$
H_{\alpha,\beta,\lambda}(p) = - \int p(x) \log p(x)\, \dd x -\int V_{\alpha,\beta,\lambda}(x)p(x)\, \dd x
$$
among all the probability density functions $p$ on $(0,\infty)$. 
This statement is slightly different from the original one \cite[Theorem 2]{KI03}, and for the reader's convenience a short proof is given below. The proof is a  straightforward application of the Gibbs' inequality
\begin{equation}\label{Gibbs}
-\int p(x)\log p(x)\,\dd x \leq -\int p(x)\log q(x)\,\dd x,  
\end{equation}
for all probability density functions $p$ and $q$, say on $(0,\infty)$. Taking $q$ to be the density \eqref{C-GIG} of the classical GIG distribution and computing $\log q(x)$, we obtain the inequality 
\begin{equation}\label{C-entropy}
H_{\alpha,\beta,\lambda}(p) \leq -\log \frac{(\alpha/\beta)^{\lambda/2}}{2K_\lambda(2\sqrt{\alpha\beta})}. 
\end{equation}
Since the Gibbs inequality \eqref{Gibbs} becomes equality if and only if $p=q$, the equality in \eqref{C-entropy} holds if and only if $p=q$, as well.  

\begin{rem}
From the above observation, it is tempting to investigate the map
$$
C e^{-V(x)}\, \dd x \mapsto \text{~the maximizer $\mu_V$ of the free entropy functional $I_V$ with potential $V$},  
$$
where $C>0$ is a normalizing constant. Under some assumption on $V$, the free entropy functional $I_V$ is known to have a unique maximizer (see \cite{ST97}) and so the above map is well defined. Note that the density function $C e^{-V(x)}$ is the maximizer of the classical entropy functional with potential $V$, which follows from the same arguments as above. This map sends Gaussian to semicircle, gamma to free Poisson (when $\lambda \geq1$), and GIG to free GIG. More examples can be found in \cite{ST97}. 
\end{rem}

%%%%%%%%%%%%%%%%%%
%%%%%%%%%%%%%%%%%%
\section*{Acknowledgement}
The authors would like to thank BIRS, Banff, Canada for hospitality during the workshop ``Analytic versus Combinatorial in Free Probability'' where we started to work on this project. TH was supported by JSPS Grant-in-Aid for Young Scientists (B) 15K17549. 
KSz was partially supported by the NCN (National Science
Center) grant 2016/21/B/ST1/00005.

%%%%%%%%%%%%%%%%%%%
%%%%%%%%%%%%%%%%%%%%

\end{document}